\documentclass{amsart}
\usepackage{amssymb}
\usepackage{color}
\usepackage{float}
\usepackage[all,cmtip]{xy}
\usepackage{graphicx}
\usepackage{hyperref}
\usepackage{cancel}

\newcommand{\abs}[1]{\lvert#1\rvert}

\newcommand{\Acal}{{\mathcal{A}}}

\newcommand{\Ecal}{{\mathcal{E}}}

\newcommand{\Hcal}{{\mathcal{H}}}

\newcommand{\CC}{\mathbb{C}}

\newcommand{\FF}{\mathbb{F}}

\newcommand{\HH}{\mathbb{H}}

\newcommand{\NN}{\mathbb{N}}

\newcommand{\QQ}{\mathbb{Q}}
\newcommand{\RR}{\mathbb{R}}

\newcommand{\ZZ}{\mathbb{Z}}

\newcommand{\RFH}{\operatorname{RFH}}
\newcommand{\RFC}{\operatorname{RFC}}
\newcommand{\HM}{\operatorname{HM}}

\newcommand{\Crit}{\operatorname{Crit}}

\newtheorem{theorem}{Theorem}
\newtheorem*{theorem*}{Theorem}
\newtheorem{proposition}{Proposition}[section]
\newtheorem{lemma}[proposition]{Lemma}

\theoremstyle{definition}

\theoremstyle{remark}
\newtheorem{remark}[proposition]{Remark}

\numberwithin{equation}{section}

\begin{document}

\title{The Bott--Samelson theorem for positive Legendrian isotopies}

\author{Lucas Dahinden}
\address{Universit\'e de Neuch\^atel (UNINE)}
\curraddr{Institut de Math\'ematiques, Rue Emile-Argand 2, 2000 Neuch\^atel}
\email{l.dahinden@gmail.com}

\date{\today}

\begin{abstract}
	The classical Bott--Samelson theorem states that if on a Riemannian manifold all geodesics issuing from a certain point return to this point, then the universal cover of the manifold has the cohomology ring of a compact rank one symmetric space. This result on geodesic flows has been generalized to Reeb flows and partially to positive Legendrian isotopies by Frauenfelder--Labrousse--Schlenk. We prove the full theorem for positive Legendrian isotopies. 
\end{abstract}

\maketitle

\section{Introduction and result}\label{sec:intro}


The spherization $S^*Q$ of a manifold $Q$ is the space of positive line elements in the cotangent bundle $T^*Q$. The tautological one-form $\lambda$ on $T^*Q$ does not pass to the quotient, but its kernel does. This endows $S^*Q$ with a cooriented contact structure~$\xi$. 


Let $j_t:L\to S^*Q$ be a smooth family of embeddings such that $j_t(L)$ is a Legendrian submanifold of $S^*Q$ for all $t$. Then $L_t=j_t(L)$ is called a Legendrian isotopy. If  $\alpha(\frac d{dt} j_t(x))>0$ for one and hence any coorientation preserving contact form $\alpha$ for $\xi$ and all $x\in L$, then $L_t$ is called positive. Frauenfelder--Labrousse--Schlenk proved the following Theorem.

\begin{theorem} \cite[Theorem 2.13]{FLS}\label{BS1} Let $Q$ be a closed connected manifold of dimension $\geq 2$. Suppose there exists a positive Legendrian isotopy $L_t$ in the spherization $S^*Q$ that connects the fiber over a point with itself, i.e.\ $L_0=L_1=S^*_qQ$. Then the fundamental group of $Q$ is finite and the integral cohomology ring of the universal cover of $Q$ is generated by one element.
\end{theorem}

	We note that by a deep result in algebraic topology, a manifold with integral cohomology ring generated by one element is homotopy equivalent to $S^n$, $\RR P^n$ or $\CC P^n$ or has the integral cohomology ring of $\HH P^n$ or the Cayley plane, see~\cite{Be78} and the references therein.

In this paper we prove the following addition to Theorem~\ref{BS1}, which was conjectured in~\cite{FLS}.

\begin{theorem}\label{main}
	Under the assumptions of Theorem~\ref{BS1}, if furthermore $L_t\cap L_{0}=\emptyset$ for $0<t<1$, then $Q$ is simply connected or homotopy equivalent to $\RR P^n$.
\end{theorem}

	The union of these two theorems is the complete generalization of the classical Bott--Samelson theorem from geodesic flows to positive Legendrian isotopies.


The first versions of the Bott--Samelson theorem were for geodesic flows and used Morse theory of the energy functional on the based loop space, see~\cite{Bo54},~\cite{Sa63} and~\cite{Be78}. Frauenfelder, Labrousse and Schlenk~\cite{FLS} proved versions of Theorem~\ref{BS1} and~\ref{main} for autonomous Reeb flows, using Rabinowitz--Floer homology. They also proved Theorem~\ref{BS1} using Rabinowitz--Floer homology for positive Legendrian isotopies as stated above. The puzzle piece missing in~\cite{FLS} to generalize Theorem~\ref{main} from autonomous Reeb flows to positive Legendrian isotopies is the fact that the action functional in the construction is Morse--Bott. We provide this in Lemma~\ref{morsebott}, and thus complete the proof in~\cite{FLS}. The key ingredient is the choice of Hamiltonian, which is elaborated in Lemma~\ref{specialham}. We cannot avoid the Hamiltonian to be time-dependent, but we can control the time-dependence along the Legendrian isotopy. At critical points, the resulting action functional then behaves like in the autonomous case. This paper is heavily based on~\cite{FLS}, which also contains an extensive introduction to the topic.

\subsection*{Acknowledgements}
	I wish to thank Felix Schlenk and the anonymous referee for their valuable suggestions. This work is supported by SNF grant 200021-163419/1.

\section{Recollections}\label{sec:tce}

The Rabinowitz--Floer homology we use depends on a time-dependent Reeb flow, not on a Legendrian isotopy. We first explain how we choose such a flow that restricts to a given Legendrian isotopy. Then we briefly present the version of Rabinowitz--Floer homology we use and discuss its properties. We only sketch the proofs, since they are contained in or  are analguous to proofs in~\cite{AF12,CF09,CFO,FLS}. For a general exposition of Morse--Bott homology we refer the reader to the Appendix of~\cite{Fra04}.


\subsubsection*{The choice of flow} Let $j_t:L\hookrightarrow M,\; {t\in[0,1]}$, be a positive Legendrian isotopy in a cooriented exact contact manifold $(M,\alpha)$. We denote $L_t=j_t(L)$. By the Legendrian isotopy extension theorem, see for example~\cite[Theorem 2.6.2]{Ge08}, there exists a positive contact isotopy~$\psi^t$ of $M$ such that $\psi^t(L_0)=L_t$. If furthermore $L_0=L_1$, then there exists a positive and twisted periodic (that is $\varphi^t=\varphi^{t-k}\circ\varphi^k$ for all $t\in\RR, k\in\ZZ$) contact isotopy $\varphi^t$ such that $\varphi^k(L_0)=L_0$ for all $k\in\NN$, see~\cite[Proposition 6.2]{FLS}. This isotopy is generated by a contact Hamiltonian $h^t$ that is a convex combination of the contact Hamiltonian of $\psi^t$ and the one of the Reeb flow $\psi^t_R$ (namely $h\equiv 1$), such that for $t$ near $0$ or $1$, $\varphi^t$ coincides with $\psi_R^t$. Note that in general $\varphi^t(L_0)\neq \psi^t(L_0)$ for $t\notin\NN$.


\begin{lemma}\label{specialham}
	Given a periodic Legendrian isotopy $L_t$ that is the restriction of the Reeb flow generated by the contact Hamiltonian $h\equiv 1$ for $t$ near $0$ and $1$ (as given by~\cite[Proposition 6.2]{FLS}), then the corresponding twisted periodic positive contact isotopy $\varphi^t$ can be chosen such that the time-dependent contact Hamiltonian $h^t$ that generates $\varphi^t$ satisfies $\dot h^t=0$ along $L_t$.
\end{lemma}
\begin{proof}
	The construction of $h^t$ is performed as in~\cite[Theorem~2.6.2]{Ge08}. We emphasize for a function $h^t$ and a path $\gamma(t)$ the distinction between $(\frac d{dt} h^t)(\gamma(t))$ and $\frac d{dt}(h^t(\gamma(t)))$ by using the notation $\dot h^t:=\frac d{dt} h^t$.
	
	Recall that a contact Hamiltonian $h^t$ and a contact vector field $X_t$ determine each other through the equations $h^t=\alpha(X_t)$ and $\iota_{X_t}d\alpha=dh^t(R_\alpha)\alpha - dh^t$. We define the 1-jet of $h^t$ along $L_t$  as follows. 
	\begin{eqnarray}
		h^t(j_t(x))&=&\alpha\left(\frac d{dt}j_t(x)\right)\quad\forall x\in L,\label{a1}\\
		dh^t(v)&=&-\iota_{\frac d{dt}j_t(x)}d\alpha(v)\quad \forall v\in\xi\vert_{L_t},\label{a2}\\
		dh^t\left(\frac d{dt}j_t(x)\right)&=&\frac d{dt}\left(h^t(j_t(x))\right)\quad\forall x\in L.\label{a3}
	\end{eqnarray}
	Any Hamiltonian $h^t$ that satisfies the first two equations generates a vector field $X_t$ such that $X_t(j_t(x))=\frac d{dt}j_t(x)$ for all $x\in L$. Equation~(\ref{a2}) holds for all $v\in TL_t$ since $TL_t\subseteq\xi|_{L_t}$. Equation~(\ref{a3}) does not contradict~(\ref{a2}) since $\frac d{dt}j_t(x)$ is positively transverse to $\xi$ for all $x\in L$. (Here we differ from~\cite{Ge08} where the choice in~(\ref{a3}) is $dh^t(R_\alpha)=0$.) Since $\frac d{dt}\left(h^t(j_t(x))\right)=\dot h^{t}(j_{t}(x))+dh^t(\frac d{dt} j_t(x))$ for all $x\in L$, equation~(\ref{a3}) implies $\dot h^t=0$ along $L_t$. We extend $h^t$ to a neighbourhood of $L_t$ by identifying a neighbourhood of $L_t$ with the normal bundle $NL_t\to L_t$ and choosing $h_t$ linear on each fiber.
	
Finally we extend $h^t$ to a positive function that is constant 1 outside a neighbourhood of $L_t$. Since the Legendrian isotopy is the restriction of the Reeb flow generated by $h\equiv 1$ for $t$ near $0$ and $1$, the function $h^t$ thus constructed satisfies $h^t\equiv 1$ for $t$ near $0$ and $1$, and admits a 1-periodic extension.
\end{proof}                                                                                                   


The spherization $(S^*Q,\xi)$ of a manifold $Q$ is represented by any fiberwise starshaped hypersurface $\Sigma\subset T^*Q$ in the cotangent bundle with contact structure $\ker\lambda|_\Sigma$. The map that sends a positive line element to its intersection with~$\Sigma$ is a contactormorphism. The radial dilation of a fiberwise starshaped hypersurface by a positive function is a contactomorphism onto its image. Every cooriented contact form of $(S^*Q,\xi)$ is realized as $\lambda|_\Sigma$ for some fiberwise starshaped hypersurface $\Sigma$. We choose a Riemannian metric $g$ on $Q$ and represent $S^*Q$ henceforth as the unit cosphere bundle with respect to this metric.
	With $\alpha=\lambda|_{\Sigma}$, the symplectization $(\Sigma\times \RR_{>0},d(r\alpha))$ is naturally symplectomorphic to $T^*Q\backslash Q$. A contact isotopy~$\varphi^t_\Sigma$ of $\Sigma$ admits a lift to a Hamiltonian isotopy~$\varphi^t$ of $\Sigma\times\RR_{>0}$, defined by~$\varphi^t(x,r)=(\varphi^t_\Sigma(x),\frac r{\rho_t(x)})$ where $\rho_t(x)$ is defined by~$(\varphi^t_\Sigma)^*(\alpha)|_x=\rho_t(x)\alpha|_x$, see~\cite[Proposition 2.3]{AF12}. If $\varphi^t_\Sigma$ is generated by the contact Hamiltonian $h^t$, then~$\varphi^t$ is generated by the Hamiltonian $H^t=rh^t$.
	
	
\subsubsection*{The functional} Let $h^t$ be a positive, periodic contact Hamiltonian on $(\Sigma,\ker\lambda)$. Following~\cite{AF12} we choose a lift of the contact isotopy $\varphi^t$ of $\Sigma$ generated by $h^t$ to the symplectization $(\Sigma\times \RR_{>0},d(r\alpha))$, depending on parameters $\kappa\geq 2, R\geq2$ and constants $c,C$ such that uniformly $0<c<h^t<C$. We define $\widetilde H^t=rh^t-\kappa$. The Hamiltonian $H^t$ is a deformation of $\widetilde H^t$ such that $H^t=cr-\kappa$ for $r\leq1$, $H^t=\widetilde H^t$ for $2\leq r\leq \kappa R-1$ and $H^t=Cr-\kappa$ for $r>\kappa R$. This has the effect that $H^t$ induces reparametrized $g$-geodesic flows for $r\in(0,1]\cup[\kappa R,\infty)$, and a lift of the $h^t$-contact flow for $r\in[2,\kappa R-1]$.
	
Denote by $\Omega_{T_q^*Q}T^*Q$ the set of $W^{1,2}$ paths $x:[0,1]\to T^*Q$ such that $x(0),x(1)\in T^*_qQ$. Define the functional $\Acal:\Omega_{T_q^*Q}T^*Q\times\RR\to\RR$ by
\begin{eqnarray*}
	\Acal(x,\eta)&=&\frac1\kappa\left(\int_0^1 \left[\lambda(\dot x) -\eta H^{\eta t}(x(t))\right]\; dt\right).
\end{eqnarray*}
This functional depends of course on $h^t$, but also on the parameters $\kappa,R$ and the constants $c,C$. A pair $(x,\eta)$ is a critical point of $\Acal$ if and only if $\dot x=\eta X_{H^{\eta t}}$ and $\int_0^1 H^{\eta t}(x(t))+\eta t\dot H^{\eta t}(x(t))\;dt=0$. This is equivalent to 
\begin{equation}\label{crit1}
\left\{
\begin{aligned}
		\dot x &= \eta X_{H^{\eta t}},\\
		H^{\eta}(x(1))&=0,
\end{aligned}
\right.
\end{equation}
 as one sees by using that $\eta X_{H^{\eta t}}$-chords satisfy $\frac d{dt}H^{\eta t}(x(t))=\eta\dot H^{\eta t}(x(t))$ and by integration by parts. Note that the factor $\frac 1\kappa$ does not change the critical point equations~(\ref{crit1}), but only the critical values. In fact, Lemma~\ref{peskydetails} below shows that this factor normalizes the action for critical points in such a way that the action spectrum is independent of $\kappa$.
\begin{remark}
	For autonomous Hamiltonians $H^t=H$ the second equation of~(\ref{crit1}) becomes $H(x(t))=0\;\forall\; t\in[0,1]$. Thus the critical points are flow lines on the hypersurface $H^{-1}(0)$. This hypersurface is not well-defined for time-dependent Reeb flows and $H^{\eta t}(x(t))$ might be very large for $t\neq 1$. We deal with this defect through the parameters $\kappa,R$. Intuitively speaking, the parameters create safe space ($\kappa$ towards the zero section, and $R$ towards infinity), where critical orbits are free to roam. This is made precise in the next lemma from~\cite[Proposition 4.3, Corollary 4.4]{AF12}.
\end{remark}

\begin{lemma}\label{peskydetails}
	For all $a<b$ there exist constants $\kappa_0\geq2, R_0\geq2$ such that for $\kappa\geq\kappa_0$ and $R\geq R_0$, all critical points $(x,\eta)$ of $\Acal$ with action between $a$ and $b$ satisfy $2\leq\abs{x(t)}_g\leq\kappa R-1$ for all $t$ and $\Acal(x,\eta)=\eta$. As a consequence, the critical point equation~(\ref{crit1}) and the action values are independent of the choice of $\kappa\geq \kappa_0,R\geq R_0,c,C$.
\end{lemma}


\subsubsection*{The chain group} Assume from now on that the functional~$\Acal$ is Morse--Bott for critical points with action between $a$ and $b$. Choose in addition a Morse function $f$ on $\Crit\Acal$. Then for $b\in\RR$ we define the filtered Rabinowitz--Floer chain group $\RFC^b(\Acal)$ as the $\ZZ_2$-vector space generated by the critical points of $f$ on $\Crit\Acal$ with action~$\leq b$.


\subsubsection*{The index} The index of a critical point $c=(x,\eta)$ of $f$ on $\Crit\Acal$ is defined as follows. Let $TT^*_qQ$ be the vertical Lagrangian distribution. Denote by $\mu_{RS}(x,\eta)$ the Robbin--Salamon index of the path $d(\varphi^{\eta t})^{-1}(TT^*_{x(t)}Q)$ with respect to $TT^*_{x(0)}Q$, and by $\mu_{M}$ the Morse index of $f$ on $\Crit\Acal$, see~\cite{RoSa93}. Then the index of $c$ is defined as 
$$\mu(c)=\mu_{RS}(x,\eta)-\frac{n-1}2+\frac12\mu_{M}(c),$$
where the shift by $-\frac{n-1}2$ is introduced such that the index $\mu$ agrees with the Morse index for geodesic Hamiltonians. Denote by $\RFC_*^{>0}(\Acal)$ the chain groups graded by the index $\mu$. 


\subsubsection*{The differential} For the differential, we choose an $\omega$-compatible almost complex structure~$J=J_{t,\eta}$ on~$T^*Q$ that satisfies the following properties for $r \in[0,1]\cup[\kappa R,\infty)$, following~\cite[Chapter 3]{CFO}:
\begin{itemize}
	\item $J$ is independent of~$t,\eta$,
	\item $J$ maps $r\partial_r$ to $X_{\frac 12r^2}$ and preserves $\ker\lambda|_{\{r=const\}}$,
	\item $J$ is invariant under the Liouville flow $(y,r)\mapsto(y,e^{t}r),\;t\in\RR$.
\end{itemize} 
Define the $L^2$-metric
\begin{equation}
\langle (v_1,\eta_1),(v_2,\eta_2)\rangle_{J}=\frac1\kappa\int_0^1\omega (v_1,Jv_2)\;dt+\frac{\eta_1\eta_2}\kappa\nonumber
\end{equation}
on $\Omega_{T_q^*Q}T^*Q\times\RR$. Further, choose a Morse--Smale metric $m$ on $\Crit\Acal$. The differential of degree $-1$ is defined by the $\ZZ_2$-count of finite energy negative gradient flow lines with cascades. A flow line with cascades starts at a critical point of $f$ at time $-\infty$, then runs until a finite time as negative $m$-gradient flow line on $\Crit\Acal$, then runs as negative $\langle\cdot,\cdot\rangle_{J}$-flow line from one component of $\Crit\Acal$ to an other (from time $-\infty$ to $+\infty$), then runs for a finite time along a negative $m$-gradient flow line, $\ldots$, and after finitely many such changes (cascades) ends in a critical point of $f$ at time $+\infty$. To show that this differential is well defined and $d^2=0$, one has to show that for $\Acal(c^+),\Acal(c^-)\in[a,b]$ the moduli space of finite energy negative gradient flow lines with cascades from $c^+$ to $c^-$ is compact modulo breaking. This follows from standard arguments as soon as one has established $L^{\infty}$ bounds on the Floer strips underlying the $\langle\cdot,\cdot\rangle_{J}$-parts of the flow lines, on the derivatives of the Floer strips, and on $\eta$. The $L^\infty$ bounds on the Floer strips follow from a maximum principle since our Hamiltonian is convex for $r\notin[1,\kappa R]$. The $L^\infty$ bounds on the derivatives follow from the exactness of $\omega=d\lambda$ that prevents bubbling. The following lemma shows that for almost critical points, $\eta$ is bounded by the action.
\begin{lemma}[Fundamental Lemma]
	There exists $\varepsilon>0$ such that
	\begin{equation*}
		\|\nabla\Acal(x,\eta)\|<\varepsilon\Rightarrow \abs{\eta}\leq \frac1\varepsilon(\Acal(x,\eta)+1).\nonumber
	\end{equation*}
\end{lemma} 
	This is a version with Lagrangian boundary conditions of~\cite[Lemma 4.5]{AF12} and is proved using a by now standard scheme, see~\cite[Proposition 3.1]{CF09}. The $L^\infty$ bound on $\eta$ is then obtained as in~\cite[Corollary 3.3]{CF09}. 
	
\subsubsection*{The Homology}	We define $\RFC^b_{a,*}(\Acal)$ as the quotient chain complex $\RFC^b_*/\RFC^a_*$. By Lemma~\ref{peskydetails}, for $\kappa\geq\kappa_0,$ and $R\geq R_0$ the generators and actions of this chain complex do not depend on the choice of $\kappa,R,c,C,$ and by standard continuation arguments the resulting homology is independent up to canonical isomorphisms of a generic choice of $g,J,m$. Finally, define $\RFC_*^{>0}(\Acal)$ as the inverse direct limit $\lim_{b\nearrow\infty}\lim_{a\searrow 0}\RFC_{a,*}^b(\Acal)$ under the homomorphisms induced by inclusion and denote the resulting homology by $\RFH^{>0}_*(\Acal)$. 


\subsubsection*{Invariance} For any other twisted periodic and positive contact isotopy~$\widetilde\varphi^t$ such that the corresponding functional~$\widetilde\Acal$ is Morse--Bott, we have
$$\RFH^{>0}_*(\Acal)\cong\RFH^{>0}_*(\widetilde\Acal).$$ 
This can be shown like the invariance (29) in the proof of~\cite[Lemma 5.4]{FLS} with the additional explanation after~\cite[Lemma 5.5]{FLS}, by considering the path of Hamiltonians $H_s^t=(1-\beta(s))H^t+\beta(s)\widetilde H^t$, where $\beta(s)$ is a smooth monotone function with $\beta(s)=0$ for $s\leq0$ and $\beta(s)=1$ for $s\geq1$, generating a path of functionals $\Acal_s$ that connects $\Acal$ and $\widetilde\Acal$, where the constants $\kappa, R,c,C$ are chosen uniformly in $s$. Note that $\partial_sH_s^t$ is compactly supported, thus a continuation homomorphism can be defined. Also note that there exists an $\varepsilon>0$ such that for all $s\in\RR$ the action spectrum of $\Acal_{H_s^t}$ and the interval $(0,\varepsilon]$ are disjoint. Using this, we can exclude that critical values cross $0$ during the continuation. The isomorphism follows then by standard arguments. 

In particular for~$h^t\equiv 1$ the corresponding functional~$\Acal_g$ is the functional of the $g$-geodesic flow. Denote by $\HM^{>0}_*(\Ecal)$ the $\ZZ_2$-Morse homology relative the constant loop of the energy functional~$\Ecal(x)=\int_0^1\frac12g(\dot x,\dot x)\;dt$ on the space of based loops in $Q$. The following result is a special case of Merry's theorem~\cite[Theorem 3.16]{M}. 
$$\RFH^{>0}_*(\Acal_g)\cong\HM^{>0}_*(\Ecal).$$
Since $\HM_*^{>0}(\mathcal{E})$ is isomorphic to the homology $H_*(\Omega_q,q;\ZZ_2)$ relative the constant loop, we obtain

\begin{lemma}\label{merry}
	$\RFH^{>0}_*(\Acal_g)\cong H_*(\Omega_q,q;\ZZ_2).$
\end{lemma}

\section{Proof of Theorem~\ref{main}}\label{sec:proof}

Recall that Theorem~\ref{BS1} is shown in~\cite{FLS}. In this section we prove Theorem~\ref{main}, using the results that are already established in Theorem~\ref{BS1}. The main step is to show that in this situation the action functional is Morse--Bott. Theorem~\ref{main} then follows exactly as in~\cite{FLS}.

\begin{remark}\label{xivert}
Let $(x,\eta)$ be a critical point of $\Acal$ for $\kappa,R,c,C$ as in Lemma~\ref{peskydetails} and $h^t$ as in Lemma~\ref{specialham}. Then along $x$ we have~$H^t=\widetilde H^t=rh^t-\kappa$, and hence~$\eta\dot H^{\eta t}(x(t))=0$. Since~$\frac d{d t}H^{\eta t}(x(t))=\eta\dot H^{\eta t}(x(t))$ we thus have $H^{\eta t}(x(t))=0$ for all $t$ and $\Acal(x,\eta)=\eta$. In this sense the choice made in Lemma~\ref{specialham} is designed such that the functional that arises from the situation of Theorem~\ref{main} behaves at critical points as in the autonomous case. 
\end{remark}


\begin{lemma}\label{morsebott}
	In the situation of Theorem~\ref{main} and for $h^t$ chosen as in Lemma~\ref{specialham}, the action functional~$\Acal$ defined above is Morse--Bott at the critical sets with positive action, the components of the critical manifold being diffeomorphic to~$S^*_qQ\times \{k\}, k\in\NN$.
\end{lemma}

\begin{proof}
A diffeomorphism from the critical manifolds to~$S^*_qQ\times \{k\}$ is given by mapping critical points $(x,k)\in\Omega_qQ\times\RR$ to $(x(1),k)\in T^*_qQ\times\{k\}$. Since $h^t$ is constant $1$ for $t$ near $k\in\NN$ and by the equations~(\ref{crit1}), the image of this map is~$\{(q,p)\in T^*Q\mid\abs p_g=1\}\times\{k\}\cong S^*_qQ\times \{k\}$.

The functional $\Acal$ is Morse--Bott if the kernel of the Hessian $\Hcal\Acal$ is exactly the tangent space of the critical manifold. The inclusion $T\Crit\Acal\subseteq\ker\Hcal\Acal$ is obvious, we will show the converse. A tangent vector to~$x\in\Omega_{q}Q$ is a section $\hat x$ of the pullback bundle $x^*TT^*Q$. Assume that~$(\hat x,\hat\eta)\in\ker\Hcal(\Acal)$. Since~$\hat x\in T\Omega_qQ$, the endpoints of $\hat x$ are in the vertical subbundle,~$\hat x(i)\in T_{x(i)}T^*_qQ\subseteq\ker\lambda,$ for~$i=0,1$. 

We will compute $\Hcal\Acal((\hat x,\hat\eta),(\check x,\check \eta))$ where $(\check x,\check\eta)$ is another vector based at~$(x,\eta)$. Assume for the moment that $x$ lies in a single Darboux chart and that in local coordinates we have $x=(q,p)$ and $\hat x=(\hat q,\hat p)$. As a preparation we compute
\begin{eqnarray*}
	d\left(\int_0^1\lambda(\dot x)\;dt\right)(\hat x)&=&\int_0^1\frac d{d\epsilon}(p+\epsilon\hat p)(\dot q+\epsilon\dot{\hat q})\mid_{\epsilon=0}\;dt\\
	&=&\int_0^1 \hat p\dot q+p\dot{\hat q}\;dt\\
	&=&\int_0^1 \hat p\dot q-\dot p\hat q \;dt +p\hat q\mid_0^1\\
	&=&\int_0^1 \omega (\hat x,\dot x)\;dt,
\end{eqnarray*}
where the last equality holds because the endpoints of $\hat x$ lie in the vertical subbundle and thus $\hat q(i)=0$ for $i=0,1$. If $x$ does not lie in a single chart, the same follows after finitely many coordinate changes. Similarly we compute
$$d\left(\int_0^1\omega(\hat x,\dot x)\;dt\right)(\check x)=\int_0^1 \omega(\hat x,\dot{\check x})\;dt= \int_0^1 \omega(\check x,\dot {\hat x})\;dt+\omega(\hat x,\check x)|_0^1=\int_0^1 \omega(\check x,\dot {\hat x})\;dt,$$
where the last equality holds because the endpoints of $\hat x$ and $\check x$ lie in the vertical subbundle which is a Lagrangian. Using these preparations we can now compute
\begin{eqnarray*}
	\Acal(x,\eta) &=& \int_0^1\lambda(\dot x) - \eta H^{\eta t}(x(t))\;dt,\\
	d\Acal(\hat x,\hat\eta) &=& \int_0^1 \omega(\hat x,\dot x) - \eta dH^{\eta t}(\hat x) -\; \hat \eta \left(H^{\eta t}(x(t)) + \eta t \dot H^{\eta t}(x(t))\right)\;dt, \\ 
	\Hcal \Acal((\hat x,\hat\eta),(\check x,\check\eta)) &=& \int_0^1 \omega(\check x,\dot {\hat x}) - \hat \eta\left( dH^{\eta t}(\check x) + \eta td\dot H^{\eta t}(\check x)\right)\\	
&& -\;\check\eta\left( dH^{\eta t}(\hat x)+\eta t\;d\dot H^{\eta t}(\hat x) +\hat\eta \big(2t\dot H^{\eta t} + \eta t^2 \ddot H^{\eta t} \big)\right)\;dt.
\end{eqnarray*}

Thus $(\hat x,\hat \eta)$ lies in $\ker\Hcal\Acal$ if and only if the following equations are satisfied.
\begin{eqnarray}
	\dot {\hat x} &=& \hat \eta \big(X_{H^{\eta t}} + \eta t X_{\dot H^{\eta t}}\big)\qquad \forall t,\label{first0}\\
	0 &=& \int_0^1 dH^{\eta t}(\hat x) + \eta t d\dot H^{\eta t}(\hat x) + \hat \eta (2t\dot H^{\eta t} + \eta t^2\ddot H^{\eta t})\;dt.\label{second0}
\end{eqnarray}

We translate these equations to the fixed vector space~$T_{x(0)}T^*Q$ by pulling back along $\varphi^{\eta t}$: Define
\begin{eqnarray*}
	v(t) &=& D\varphi^{-1}\hat x(t),
\end{eqnarray*}
where we abbreviated $\varphi^{\eta t}$ to $\varphi$ for better readability. Since $\varphi$ is a symplectomorphism, $D\varphi^{-1}X_{H^{\eta t}}=X_{\varphi^{*}H^{\eta t}}$. Thus equation~(\ref{first0}) becomes 
\begin{eqnarray}
	\dot v &=& \hat \eta (X_{\varphi^{*}H^{\eta t}}+ \eta t X_{\varphi^{*}\dot H^{\eta t}}) \qquad \forall t,\label{first1}
\end{eqnarray}

Integrating equation~(\ref{first1}), we obtain
\begin{eqnarray*}
	v(1)&=&v(0)+\hat\eta \int_0^1 X_{\varphi^{*}H^{\eta t}} +\eta t X_{\varphi^{*}\dot H^{\eta t}}\; dt.
\end{eqnarray*}
Since $H^{\eta t}$ is (after addition of the constant $\kappa$) 1-homogeneous in the fibers near $\Crit\Acal$, the flow $\varphi^{\eta t}$ commutes with dilations by a factor close to $1$. Thus also~$\varphi^*H^{\eta t}$ is (after addition of $\kappa$) 1-homogeneous, thus~$\varphi^*\dot H^{\eta t}$ is 1-homogeneous and so near $\Crit\mathcal A$, $X_{\varphi^*\dot H^{\eta t}}$ is a lift of the contact Hamiltonian vector field $X_{\varphi^*\dot h^{\eta t}}$ on the spherization $S^*Q$. For $h^t$ chosen as in Lemma~\ref{specialham}, we have $\varphi^*\dot h^{\eta t}=0$. Thus~$X_{\varphi^*\dot h^{\eta t}}$ lies in the contact structure,~$X_{\varphi^*\dot h^{\eta t}}\in\ker\lambda|_{S^*Q}$, and thus~$X_{\varphi^* \dot H^{\eta t}}\in\ker\lambda|_{S^*Q}\oplus\langle\partial_r\rangle=\ker\lambda$. By the geometric setup of the theorem,~$D\varphi^{-1}T_{x(1)}T^*_qQ=T_{x(0)}T^*_qQ$, so with $\hat x(i)$ also the endpoints of $v$ lie in the vertical subbundle,~$v(i)\in T_{x(0)}T^*_qQ\subseteq \ker\lambda$. Thus we conclude that~$\hat \eta\int_0^1 X_{\varphi^{*}H^{\eta t}} dt\in \ker\lambda$. But~$X_{\varphi^{*}H^{\eta t}}\pitchfork_+\ker\lambda$ for all $t$ since $h^{\eta t}$ is positive, and thus~$\int_0^1 X_{\varphi^{*}H^{\eta t}} dt\pitchfork_+ \ker\lambda$. We conclude that $\hat\eta=0$ and with~(\ref{first1}) that $v$ is constant.

Recall that our task is to show that~$(\hat x,\hat\eta)=(\hat x,0)\in T\Crit\Acal$, and recall from~(\ref{crit1}) that~$\Crit\Acal=\{x\mid\dot x(t)=\eta X_{H^{\eta t}}\}\times\NN\cap \{x\mid H^{\eta}(x(1))=0\}\times\NN$. We first define the path $(x_s,\eta_s)\in \{x\mid\dot x(t)=\eta X_{H^{\eta t}}\}\times\NN$ by $x_s(t)=\varphi^{\eta t}(x(0)+s v),\eta_s\equiv\eta$. Then $\frac d{ds}(x_s,\eta_s)\mid_{s=0}=(\hat x,0)$. Thus,
\begin{eqnarray*}
(\hat x,0)&\in& T(\{x\mid\dot x(t)=\eta X_{H^{\eta t}}\}\times\NN).
\end{eqnarray*}

Since $\dot x_s=\eta X_{H^{\eta t}}(x_s)$ for all $s$, $\frac d{dt}H^{\eta t}(x_s(t))=\eta \dot H^{\eta t}(x_s(t))$ and thus also $\frac d{dt} dH^{\eta t}_{x_0}(\hat x)=\eta d\dot H^{\eta t}_{x_0}(\hat x)$. Together with $\hat \eta=0$, equation~(\ref{second0}) becomes

\begin{eqnarray*}
0 &=& \int_0^1 dH^{\eta t}(\hat x)+t\frac d{dt}dH^{\eta t}(\hat x)\;dt\\
&\stackrel{\int by\, parts}=& \int_0^1 dH^{\eta t}(\hat x)-dH^{\eta t}(\hat x)\;dt +1\cdot dH^{\eta1}(\hat x(1))-0\cdot dH^{\eta0}(\hat x(0))\\
&=& dH^{\eta}(\hat x(1)).
\end{eqnarray*}

Thus,
\begin{eqnarray*}
(\hat x,0)&\in& T(\{x\mid\dot x(t)=\eta X_{H^{\eta t}}\}\times\NN)\cap T(\{x\mid H^{\eta}(x(1))=0\}\times\NN) =T\Crit\Acal,
\end{eqnarray*}
as claimed. 
\end{proof}

Before we can continue we need two observations about the index. Since the components of $\Crit\Acal$ are spheres $S^*_qQ\times \{k\}$, the Morse function $f$ on $\Crit\Acal$ can be chosen with exactly two critical points $c^-_{k},c^+_{k}$ per component, with Morse index $0$ and $d-1$.

\begin{lemma}\label{muzero}
	The Robbin--Salamon index of $(x(t),k)\in\Crit\Acal$ depends only on $k$ and is equal to $k\mu_0$ for some constant $\mu_0\geq1$.
\end{lemma}
\begin{proof}
The proof goes exactly as in~\cite[Section 5.2]{FLS} and uses Rabinowitz--Floer homology over $\ZZ$ coefficients, which is developed in~\cite{FLS} to prove Theorem~\ref{BS1}. We repeat the argument without developing the theory over $\ZZ$ coefficients and refer the interested reader to~\cite{FLS}. Note that the change of coefficients changes neither the critical point equation nor the index.

The subset of $\Crit\Acal$ with $\eta=k$ is a sphere and thus connected. Let $(x_0,k),(x_1,k)$ be two critical points of $\Acal$ and $(x_s,k)$ be a path in $\Crit\Acal$ connecting them. Identify the vector spaces $T_{x_s(0)}T^*Q$ in such a way that $TT^*_{x_s(0)}Q$ is constant. Then $d(\varphi^{kt})^{-1}(TT^*_{x_s(t)}Q)$ is a homotopy with parameter $s$ with constant endpoints of paths with parameter $t$ of Lagrangian subspaces. Thus the two paths $d(\varphi^{kt})^{-1}(TT^*_{x_0(t)}Q)$ and $d(\varphi^{kt})^{-1}(TT^*_{x_1(t)}Q)$ are stratum homotopic in the sense of~\cite{RoSa93} and thus $\mu_{\rm RS}(x_0,k)=\mu_{\rm RS}(x_1,k)$. We conclude that the Robbin--Salamon index only depends on $k$. Since every $\varphi^{kt}$ flow line is the $k$-fold concatenation of $\varphi^t$ flow lines, $\mu_{\rm RS}(\varphi^{kt}x(0),k)=k\mu_{\rm RS}(\varphi^tx(0),1)=:k\mu_0$ by the concatenation property of the Robbin--Salamon index.

Assume that $\mu_0\leq0$. Since the signatures of $c_i^\pm$ are $\pm\frac{d-1}2$ (in particular bounded), there exists a $k_0$ such that $\mu(c)< k_0\;\forall c\in\Crit\Acal$. Thus for $k\geq k_0$ we have by deformation of $\Acal$ to a geodesic functional $\Acal_g$, and by the $\ZZ$-version of Lemma~\ref{merry} (also contained in~\cite{M}),  
$$0=\RFH_k^{>0}(\Acal;\ZZ)\cong \RFH_k^{>0}(\Acal_g;\ZZ)\cong H_k(\Omega_q Q,q;\ZZ)\cong \widetilde H_k(\Omega_q Q;\ZZ),$$ 
and thus also $H_k(\Omega_q\widetilde Q;\ZZ)\cong 0$. Thus for all $k\geq k_0+1$ and $\FF=\ZZ_p$ for any prime number $p$ or $\FF=\QQ$ we have $H_*(\Omega_q\widetilde Q,q;\FF)\cong 0$. By~\cite[Proposition 10]{Se51} this implies $H_k(\widetilde Q,\FF)\cong 0$ for all $k\geq 1$ and for all $\FF=\ZZ_p$ or $\QQ$ and thus $\widetilde Q$ is contractible. Since $\dim Q\geq 2$ and $Q$ is closed, we must have $\abs{\pi_1(Q)}=\infty$ which contradicts Theorem~\ref{BS1}.
\end{proof}

With this our main theorem follows exactly as in~\cite{FLS}. We repeat the proof for the convenience of the reader.

\begin{proof}[Proof of Theorem~\ref{main}]
	The chain group $\RFC^{>0}_*(\Acal)$ of the Rabinowitz--Floer homology is generated by the critical points $c^{\pm}_k,\; k\geq 1$, where 
	$$\mu(c^{-}_k)=k\mu_0 - (d-1),\quad\mu(c^{+}_k)=k\mu_0.$$
	By Lemma~\ref{muzero}, $\mu_0\geq 1$. Hence there is one critical point of index zero if $\mu_0$ is a divisor of $(d-1)$ and no critical point of index zero otherwise. Hence after a deformation to the functional $\Acal_g$ of a geodesic flow, with Lemma~\ref{merry} and the reduced long exact $\ZZ_2$-homology sequence of the pair $(\Omega_qQ,q)$ we find that 
	$$\RFH_0^{>0}(\Acal)\cong \RFH_0^{>0}(\Acal_g)\cong H_0(\Omega_q Q,q;\ZZ_2)\cong \widetilde H_0(\Omega_q Q;\ZZ_2)$$
	is $0$ or $\ZZ_2$, thus $\pi_1(Q)$ is $0$ or $\ZZ_2$. In the first case we are done, so assume the second case.	By Theorem~\ref{BS1}, $Q$ is a closed manifold such that $H^*(\widetilde Q;\ZZ_2)$ is generated by one element. Then by~\cite[Corollary 3.8]{FraSch06}, $Q$ is either homotopy equivalent to $\RR P^d$, or $\widetilde Q$ is homotopy equivalent to $\CC P^{2n+1}$. In the former case we are done, so assume the latter. 
	
	We denote $\dim(\CC P^{2n+1})=2(2n+1)=d$. Assume first that $\mu_0\geq 2$, then $\mu(c_1^-)=\mu_0-d+1\leq \mu(c)-2$ for all other critical points $c$. This means that $c_1^-$ is the lowest index generator of $\RFH^{>0}_*(\Acal)$. The lowest index non-vanishing group is $\RFH^{>0}_0(\Acal)\cong \widetilde H_0(\Omega_qQ;\ZZ_2)\cong\ZZ_2$ and thus $\mu(c_1^-)=0$ and $\mu_0=d-1$. Recall that 
	\begin{equation*}
		H_*(\CC P^{2n+1};\ZZ_2)=
			\begin{cases}
				\ZZ_2 \quad\mbox{if }*=kd\mbox{ or }kd+1\mbox{ for }k\in\NN_0,\\
				0 \quad\mbox{otherwise.}
			\end{cases}		
	\end{equation*}
Since $\Omega_qQ$ is homotopy equivalent to the disjoint union of two copies of $\Omega_q\CC P^{2n+1}$, 
$$H_*(\Omega_qQ;\ZZ_2)\cong H_*(\Omega_q\CC P^{2n+1})\oplus H_*(\Omega_q\CC P^{2n+1}).$$
In particular $\ZZ_2\oplus\ZZ_2\cong H_{2d}(\Omega_qQ;\ZZ_2)\cong \widetilde H_{2d}(\Omega_qQ;\ZZ_2)\cong\RFH_{2d}^{>0}(\Acal)$, which is only possible if we have two generators of index $2d$. The $c^+_k$ have pairwise different indices and so do the $c^-_k$, thus there must be $k^+$ and $k^-$ such that $\mu(c^+_{k^+})=k^+\mu_0-d+1=\mu(c^-_{k^-})=k^-\mu_0=2d$. But since $\mu_0=d-1$ and $d\geq6$, this is impossible.

	The remaining case is $\mu_0=1$. Since $d\geq 6$, the critical points with negative index are exactly
$$\mu(c^-_1)=-d+2,\:\mu(c^-_2)=-d+3,\:\ldots,\:\mu(c^-_{d-2})=-1.$$
If the chord underlying $c^-_1$ were contractible in $\Omega_{T^*_qQ}T^*Q$, then all chords underlying $c^\pm_k$ would be contractible since they all are concatenations of chords homotopic to the chord underlying $c^-_1$. This contradicts the fact that they also generate the $\ZZ_2$-homology of the connected component of noncontractible chords in $\Omega_qQ$. Thus the chord underlying $c_1^-$ must be noncontractible. Since $\pi_1(Q)=\ZZ_2$ and since the chord underlying $c_2^-$ is the concatenation of two chords homotopic to the chord underlying $c_1^-$, the chord underlying $c_2^-$ is contractible and in particular not homotopic to the chord underlying $c_1^-$. The boundary operator is defined by flow lines with cascades with underlying Floer strips and paths in $\Crit\Acal$, thus every chord underlying a critical point is homotopic to the chords underlying the summands of its boundary. Thus $c_1^-$ cannot contribute to the boundary of $c_2^-$. Since all other critical points have higher index, we conclude that $c_1^-$ is not a boundary. Since $c_1^-$ is in the lowest degree chain group, it is closed and hence represents a non-trivial homology class. Thus $\RFH_{-d+2}^{>0}(\Acal)\cong \widetilde H_{-d+2}(\Omega_qQ;\ZZ_2)$ does not vanish, which is impossible since $-d+2<0$.
\end{proof}


\end{document}